\documentclass[11pt,reqno]{amsart}
\usepackage[utf8]{inputenc}
\usepackage{amssymb}
\usepackage[all]{xy}
\usepackage{xcolor}
\usepackage{hyperref}
\usepackage{tikz}
\usepackage{amsmath}
\usepackage{amsfonts}

\newtheorem{thm}{Theorem}[section]
\newtheorem{lem}[thm]{Lemma}
\newtheorem{prop}[thm]{Proposition}

\newtheorem{lem-def}[thm]{Lemma-Definition}

\theoremstyle{definition}
\newtheorem{dfn}[thm]{Definition}

\theoremstyle{remark}
\newtheorem{remark}[thm]{Remark}

\newcommand{\CG}{{\mathcal{G}}}

\makeatletter
\@namedef{subjclassname@2020}{%
  \textup{2020} Mathematics Subject Classification}
\makeatother

\begin{document}

%\numberwithin{equation}{section}

\title[ Pullbacks of Groupoid $C^*$-Algebras over Disjoint Invariant Sets ]
{ Pullbacks of Groupoid $C^*$-Algebras over Disjoint Invariant Sets}

\author[Gilles G. de Castro]{Gilles G. de Castro}
\address{Departamento de Matem\'atica, Universidade Federal de Santa Catarina, 88040-970 Florian\'opolis SC, Brazil.} \email{gilles.castro@ufsc.br}

\author[E. J. Kang]{Eun Ji Kang}
\address{Research Institute of Mathematics, Seoul National University, Seoul 08826, 
Korea} \email{kkang33\-@\-snu.\-ac.\-kr}

\thanks{This research was suported by Basic Science Research Program through the 
National Research Foundation of Korea(NRF) funded by the Ministry of Education (RS-2026-25575087).}



\begin{abstract}

We establish a pullback theorem for \(C^*\)-algebras of locally compact
Hausdorff étale groupoids. The theorem shows that, when the unit space is
covered by two closed invariant subsets whose complements are disjoint open
invariant subsets, the corresponding groupoid \(C^*\)-algebras form a
pullback diagram in the category of \(\mathbb T\)-\(C^*\)-algebras and
\(\mathbb T\)-equivariant \(*\)-homomorphisms, for the gauge actions induced
by a \(\mathbb Z\)-valued cocycle and its restrictions.

We then prove a collection of boundary-path decomposition
theorems for graphs, relative graphs, and topological graphs. We show that
admissible decompositions of graphs, together with their analogues in the
relative and topological settings, induce corresponding decompositions of
boundary path spaces. Combining these decomposition theorems with the
groupoid pullback theorem, we recover previously known pullback theorems for
graph \(C^*\)-algebras, relative graph \(C^*\)-algebras, and topological
graph \(C^*\)-algebras. Thus these pullback phenomena are explained by a
single groupoid-theoretic mechanism.

\end{abstract}

\maketitle

\section{Introduction}
Pullback diagrams play an important role in \(C^*\)-algebra theory, providing
an algebraic counterpart to gluing constructions in topology. They describe
how a \(C^*\)-algebra can be assembled from compatible pieces and often give
access to computational tools such as the Mayer--Vietoris six-term exact
sequence in \(K\)-theory.

A number of pullback constructions have appeared in the literature on
graph-related \(C^*\)-algebras. For graph \(C^*\)-algebras, Hajac and
collaborators introduced admissible decompositions of graphs and showed that
they give rise to pullback diagrams of graph \(C^*\)-algebras \cite{HRT}.
Analogous results were subsequently established for relative graph
\(C^*\)-algebras \cite{BS}, topological graph \(C^*\)-algebras \cite{GQT},
and higher-rank graph \(C^*\)-algebras \cite{KPSW2016}. Although these
results have a strikingly similar form, their proofs are generally carried
out within the particular combinatorial framework of the corresponding
class of graph-like objects.

It is therefore not apparent from the existing literature whether these
pullback constructions are essentially independent phenomena or instances
of a more general principle. One of the main aims of this paper is to
identify the common mechanism underlying them.

Our first main goal is a pullback theorem for \(C^*\)-algebras of
locally compact Hausdorff \'etale groupoids. More precisely, let
\(\mathcal G\) be a locally compact Hausdorff \'etale groupoid equipped with
a continuous cocycle $c:\mathcal G\to\mathbb Z$.
We show that if the unit space is covered by two closed invariant subsets
whose complements are disjoint open invariant subsets, then the associated
groupoid \(C^*\)-algebras form a pullback diagram in the category of
\(\mathbb T\)-\(C^*\)-algebras and \(\mathbb T\)-equivariant
\(*\)-homomorphisms. The \(\mathbb T\)-actions are the gauge actions induced
by \(c\) and by its restrictions to the corresponding reduction groupoids.
To the best of our knowledge, this is the first pullback theorem formulated
at the level of \'etale groupoid \(C^*\)-algebras.

The second main goal concerns the passage from combinatorial
decompositions to groupoid decompositions. Our groupoid pullback theorem applies
when the unit space admits an appropriate closed invariant cover. For
graph-related \(C^*\)-algebras, such covers are not usually visible directly
from the underlying graph-like object; rather, they appear at the level of
the boundary path space.

Accordingly, a key technical ingredient of the paper is a collection of
boundary-path decomposition theorems. We show that suitable combinatorial
decompositions of graph objects induce corresponding decompositions of
their boundary path spaces. These results translate combinatorial gluing data
into closed invariant covers of the unit spaces of the associated \'etale
groupoids, and are of independent interest.

Taken together, the groupoid pullback theorem and the boundary-path
decomposition results explain why pullback diagrams repeatedly occur in
graph-related \(C^*\)-algebras. In particular, previously known pullback
theorems for graph \(C^*\)-algebras, relative graph \(C^*\)-algebras, and
topological graph \(C^*\)-algebras are recovered from our groupoid theorem.

Thus, the present paper is not merely a generalization of existing pullback
results. Rather, it provides a conceptual explanation for why these pullback
phenomena arise. From this perspective, pullback diagrams for graph-related
\(C^*\)-algebras arising from regular inclusions could be understood as consequences of closed invariant
covers of the unit spaces of the associated \'etale groupoids.

The paper is organized as follows. In Section~\ref{sec:preliminaries}, we
recall the necessary background on \'etale groupoids, cocycles, and groupoid
\(C^*\)-algebras. In Section~\ref{sec:main}, we establish the groupoid
pullback theorem. The remaining sections are devoted to applications, where
we prove boundary-path decomposition theorems and recover previously known
pullback results for graph \(C^*\)-algebras, relative graph \(C^*\)-algebras,
and topological graph \(C^*\)-algebras as consequences of our groupoid
pullback theorem.

\section{Preliminaries}\label{sec:preliminaries}

\subsection{Groupoids and groupoid $C^*$-algebras}

A \emph{groupoid} $\CG$ is a small category in which every morphism $\gamma \in \CG$ has a unique inverse $\gamma^{-1} \in \CG$. For $\gamma \in \CG$, the \emph{range} and \emph{source} maps are given by
\[
r(\gamma) := \gamma \gamma^{-1}, \quad s(\gamma) := \gamma^{-1} \gamma,
\]
where composition is read from right to left. The set of \emph{composable pairs} is 
\[
\CG^{(2)} := \{ (\alpha, \beta) \in \CG \times \CG \mid s(\alpha) = r(\beta) \},
\]
and the \emph{unit space} is 
\[
\CG^{(0)} := r(\CG) = s(\CG).
\]

A \emph{topological groupoid} is a groupoid endowed with a topology under which composition and inversion are continuous. 
A locally compact Hausdorff groupoid $\CG$ is \emph{\'etale} if the range and source maps $r, s: \CG \to \CG^{(0)}$ are local homeomorphisms. A subset $B \subseteq \CG$ is called a \emph{bisection} if the restrictions $r|_B$ and $s|_B$ are injective; if $B$ is open in an \'etale groupoid, these restrictions are homeomorphisms onto open subsets of $\CG^{(0)}$. Every \'etale groupoid admits a basis of open bisections.
We say that an \'etale groupoid $\CG$ is \emph{ample} if it admits
a basis of compact-open bisections.

Let $\CG$ be a groupoid and $U \subseteq \CG^{(0)}$. We define
\[
\CG_U := \{ \gamma \in \CG \mid s(\gamma) \in U \}, \quad
\CG^U := \{ \gamma \in \CG \mid r(\gamma) \in U \}, \quad
\CG_U^U := \CG_U \cap \CG^U.
\]

A subset $U \subseteq \CG^{(0)}$ is called $\CG$-\emph{invariant} if 
for every $\gamma \in \CG$, $s(\gamma) \in U$ implies $r(\gamma) \in U$. 
Equivalently, $U$ is invariant exactly when $\CG_U = \CG^U$, 
so that $\CG_U$ is itself a subgroupoid of $\CG$ with unit space $U$.

Let $\CG$ be a locally compact Hausdorff étale groupoid. The space $C_c(\CG)$ of continuous compactly supported functions on $\CG$ has a convolution product and involution defined by
\[
(f_1 * f_2)(\gamma) := \sum_{\alpha \beta = \gamma} f_1(\alpha) f_2(\beta), 
\quad f^*(\gamma) := \overline{f(\gamma^{-1})}.
\]
The \emph{full groupoid $C^*$-algebra} $C^*(\CG)$ is the completion of $C_c(\CG)$ with respect to the universal $C^*$-norm.

The following standard result for full groupoid $C^*$-algebras will be used in the proof of our main theorem.

\begin{thm}{(well-known)}\label{thm:groupoid-ses}
Let $\CG$ be a locally compact Hausdorff étale groupoid and let
$U\subseteq \CG^{(0)}$ be an open $\CG$-invariant subset.
Then, $X:=\CG^{(0)}\setminus U$ is also $\CG$-invariant, $\CG_U$ is an open étale subgroupoid of $\CG$,
and $\CG_X$ is a closed étale subgroupoid of $\CG$.
Moreover, the restriction map $\rho:C_c(\CG)\to C_c(\CG_X)$
extends to a surjective $*$-homomorphism
$
\rho:C^*(\CG)\to C^*(\CG_X)
$
with kernel $\ker\rho=C^*(\CG_U).$
Consequently, there is a short exact sequence
\[
0
\longrightarrow
C^*(\CG_U)
\longrightarrow
C^*(\CG)
\overset{\rho}{\longrightarrow}
C^*(\CG_X)
\longrightarrow
0.
\]
\end{thm}

\subsection{Gauge actions}

Let $\CG$ be a locally compact Hausdorff étale groupoid and let
$c:\CG\to\mathbb Z$ be a continuous cocycle.
The cocycle $c$ induces a strongly continuous action
$
\gamma:\mathbb T\to\operatorname{Aut}(C^*(\CG))
$
given on $C_c(\CG)$ by
$$
(\gamma_z(f))(\eta)
=
z^{c(\eta)}f(\eta)
$$
for all $z\in\mathbb T$, $f\in C_c(\CG)$ and $\eta\in\CG$.
We call $\gamma$ the \emph{gauge action} associated to $c$.

A $\mathbb T$-$C^*$-algebra is a $C^*$-algebra equipped with a strongly continuous action of $\mathbb T$.
Recall that if $(A,\alpha)$ and $(B,\beta)$ are $\mathbb T$-$C^*$-algebras, then a $*$-homomorphism $\Phi:A\to B$
is said to be {\it $\mathbb T$-equivariant} if
$$
\Phi\circ\alpha_z
=
\beta_z\circ\Phi
$$
for all $z\in\mathbb T$.

In particular, if
$c_i:\CG_i\to\mathbb Z$
are cocycles inducing gauge actions
$\gamma^{(i)}$ on
$C^*(\CG_i)$ for $i=1,2$,
then a $*$-homomorphism
$
\Phi:C^*(\CG_1)\to C^*(\CG_2)
$
is $\mathbb T$-equivariant (or \emph{gauge-equivariant}) if
$$
\Phi\circ\gamma^{(1)}_z
=
\gamma^{(2)}_z\circ\Phi
$$
for all $z\in\mathbb T$.

\section{Pullbacks of groupoid \(C^*\)-algebras}\label{sec:main}

We first recall the notion of a pullback for \(C^*\)-algebras. Suppose that
\[
A \xrightarrow{\ \mu_1\ } C \xleftarrow{\ \mu_2\ } B
\]
are \( * \)-homomorphisms between \(C^*\)-algebras.
 The corresponding pullback \(C^*\)-algebra is 
\[
A\oplus_C B
:=
\{(a,b)\in A\oplus B:\mu_1(a)=\mu_2(b)\},
\]
 equipped with the canonical coordinate projections onto
\(A\) and \(B\).
If \(A\), \(B\), and \(C\) are equipped with strongly continuous actions
\[
\alpha:\mathbb T\curvearrowright A,
\qquad
\beta:\mathbb T\curvearrowright B,
\qquad
\gamma:\mathbb T\curvearrowright C,
\]
and if \(\mu_1\) and \(\mu_2\) are \(\mathbb T\)-equivariant, then the
pullback \(A\oplus_C B\) is preserved by the componentwise action
\[
\delta_z(a,b)
=
(\alpha_z(a),\beta_z(b)).
\]
Indeed, if \((a,b)\in A\oplus_C B\), then
\[
\mu_1(\alpha_z(a))
=
\gamma_z(\mu_1(a))
=
\gamma_z(\mu_2(b))
=
\mu_2(\beta_z(b)),
\]
so that
\[
(\alpha_z(a),\beta_z(b))\in A\oplus_C B.
\]
Thus \((A\oplus_C B,\delta)\), together with the coordinate projections,
is the pullback of the diagram
\[
A \xrightarrow{\ \mu_1\ } C \xleftarrow{\ \mu_2\ } B
\]
in the category of \(\mathbb T\)-\(C^*\)-algebras and
\(\mathbb T\)-equivariant \(*\)-homomorphisms.

We now state the main pullback theorem for groupoid \(C^*\)-algebras.

\begin{thm}\label{thm:disjoint open invariant sets}
Let $\CG$ be a locally compact Hausdorff \'etale groupoid, and let
$c:\CG\to\mathbb Z$ be a continuous groupoid cocycle.
Let $U_1,U_2\subseteq\CG^{(0)}$ be open $\CG$-invariant subsets, and set
$X_i:=\CG^{(0)}\setminus U_i$ for $i=1,2$.
Assume that $U_1\cap U_2=\emptyset$. Then the diagram
\[
\begin{tikzpicture}
  \node (P) at (0,2) {$C^*(\CG)$};
  \node (A) at (-2,0) {$C^*(\CG_{X_1})$};
  \node (B) at (2,0) {$C^*(\CG_{X_2})$};
  \node (C) at (0,-2) {$C^*(\CG_{X_1\cap X_2})$};

  \draw[->] (P) -- (A) node[midway, above left] {$\rho_1$};
  \draw[->] (P) -- (B) node[midway, above right] {$\rho_2$};
  \draw[->] (A) -- (C) node[midway, below left] {$\chi_1$};
  \draw[->] (B) -- (C) node[midway, below right] {$\chi_2$};
\end{tikzpicture}
\]
is a pullback diagram with respect to the canonical restriction
homomorphisms in the category of $\mathbb T$-$C^*$-algebras and
$\mathbb T$-equivariant $*$-homomorphisms, where the
$\mathbb T$-actions are the gauge actions induced by $c$ and its
restrictions to the corresponding subgroupoids.

Equivalently, the map
\[
C^*(\CG)
\longrightarrow
C^*(\CG_{X_1})
\oplus_{C^*(\CG_{X_1\cap X_2})}
C^*(\CG_{X_2}),
\qquad
a\longmapsto\bigl(\rho_1(a),\rho_2(a)\bigr),
\]
is a $\mathbb T$-equivariant $*$-isomorphism.
\end{thm}

\begin{proof} We first note that
\[
\CG=\CG_{X_1}\cup\CG_{X_2}
\qquad\text{and}\qquad
\CG_{X_1}\cap\CG_{X_2}
=
\CG_{X_1\cap X_2}.
\]
Indeed, let $g\in\CG$. If $s(g)\in X_1$, then the invariance of
$X_1$ implies that $r(g)\in X_1$, and hence
$g\in\CG_{X_1}$. If $s(g)\notin X_1$, then $s(g)\in U_1$.
Since $U_1\cap U_2=\emptyset$, we have $s(g)\notin U_2$, and thus
$s(g)\in X_2$. The invariance of $X_2$ then implies that
$r(g)\in X_2$, so $g\in\CG_{X_2}$. Therefore,
\[
\CG=\CG_{X_1}\cup\CG_{X_2}.
\]
Moreover, by the definition of a reduction,
\[
\begin{aligned}
\CG_{X_1}\cap\CG_{X_2}
&=
\{g\in\CG:s(g),r(g)\in X_1\}
\cap
\{g\in\CG:s(g),r(g)\in X_2\}\\
&=
\{g\in\CG:s(g),r(g)\in X_1\cap X_2\}\\
&=
\CG_{X_1\cap X_2}.
\end{aligned}
\]

By Theorem \ref{thm:groupoid-ses}, we have a surjective $*$-homomorphism $\rho_i:C^*(\CG) \to C^*(\CG_{X_i})$ such that $\ker \rho_i=C^*(\CG_{U_i})$ for $i=1,2$.
Also, since $U_2 $ is an open subset of $\CG_{X_1}^{(0)}(=X_1= \CG^{(0)} \setminus U_1)$ and $\CG_{X_1 \cap X_2}=(\CG_{X_1})_{X_1 \setminus U_2}$, we have by Theorem \ref{thm:groupoid-ses} again a surjective $*$-homomorphism $\chi_1:C^*(\CG_{X_1}) \to C^*(\CG_{X_1 \cap X_2})$ with $\ker \chi_1=C^*((\CG_{X_1})_{U_2})=C^*(\CG_{U_2})$.
Noticing that  $\CG_{X_1 \cap X_2}=(\CG_{X_2})_{X_2 \setminus U_1}$,  we similarly have  a surjective $*$-homomorphism $\chi_2:C^*(\CG_{X_2}) \to C^*(\CG_{X_1 \cap X_2})$ with $\ker \chi_2=C^*((\CG_{X_2})_{U_1})=C^*(\CG_{U_1})$. The commutativity of the diagram is obvious. We claim that 
\begin{enumerate}
\item $\ker \rho_1 \cap \ker \rho_2=\{0\}$,
\item $\chi_2^{-1}(\chi_1(C^*(\CG_{X_1})))=\rho_2(C^*(\CG))$, and
\item $\rho_1(\ker(\rho_2)) = \ker(\chi_1)$.
\end{enumerate}
Indeed, for (1), note that $\ker \rho_i$ is a closed ideal in a $C^*$-algebra and that $\ker \rho_i = C^*(\CG_{U_i})$ for $i=1,2$. Hence,
\[
\ker \rho_1 \cap \ker \rho_2 = \ker \rho_1 \ker \rho_2 = C^*(\CG_{U_1})C^*(\CG_{U_2}).
\]
Choose $f_1 \in C_c(\CG_{U_1})$ and $f_2 \in C_c(\CG_{U_2})$ supported in bisections $V_1$ and $V_2$, respectively. Then
\[
f_1*f_2(g)= 
\begin{cases}
f_1(r_{V_1}^{-1}(r(g)))f_2(s_{V_2}^{-1}(s(g))) & \text{if } g \in V_1V_2,\\[2mm]
0 & \text{otherwise},
\end{cases}
\]
where $r_{V_1}$ and $s_{V_2}$ denote the restrictions of $r$ and $s$ to $V_1$ and $V_2$, respectively.
Since $U_1 \cap U_2=\emptyset$, we have $V_1V_2=\emptyset$, and hence $f_1*f_2=0$. 
Thus $C^*(\CG_{U_1})C^*(\CG_{U_2})=0$, and consequently $\ker \rho_1 \cap \ker \rho_2=\{0\}$.

For (2), we compute
\[
\chi_2^{-1}(\chi_1(C^*(\CG_{X_1})))
    =\chi_2^{-1}(C^*(\CG_{X_1 \cap X_2}))
    =C^*(\CG_{X_2})
    =\rho_2(C^*(\CG)),
\]
since $\chi_2$, $\chi_1$, and $\rho_2$ are all surjective.

Finally, for (3), since $\rho_1$ acts as the identity on $C^*(\CG_{U_2})$, we have
\[
\rho_1(\ker(\rho_2)) = \rho_1(C^*(\CG_{U_2})) = C^*(\CG_{U_2}) = \ker(\chi_1).
\]
Thus, by \cite[Proposition 3.1]{P}, the diagram is a pullback diagram.

We now check that all maps are $\mathbb T$-equivariant with respect to the
gauge actions induced by $c$. Since $c$ restricts to cocycles
\[
c_i:=c|_{\CG_{X_i}},
\qquad
c_{12}:=c|_{\CG_{X_1\cap X_2}},
\]
the gauge actions on $C^*(\CG_{X_i})$ and
$C^*(\CG_{X_1\cap X_2})$ are determined on compactly supported functions by
\[
(\gamma^{(i)}_z f)(g)=z^{c_i(g)}f(g),
\qquad
(\gamma^{(12)}_z f)(g)=z^{c_{12}(g)}f(g).
\]

For $f\in C_c(\CG)$ and $g\in \CG$, we have
\[
(\gamma_z^{(i)}\circ\rho_i)(f)(g)
=
\begin{cases}
z^{c_i(g)}f(g) & \text{if } g\in\CG_{X_i},\\[2mm]
0 & \text{otherwise},
\end{cases}
\]
while
\[
(\rho_i\circ\gamma_z)(f)(g)
=
\begin{cases}
z^{c(g)}f(g) & \text{if } g\in\CG_{X_i},\\[2mm]
0 & \text{otherwise}.
\end{cases}
\]
Since $c_i=c|_{\CG_{X_i}}$, the two expressions coincide.
Hence, we have
\[
\rho_i\circ\gamma_z
=
\gamma^{(i)}_z\circ\rho_i
\]
on $C_c(\CG)$ for $i=1,2$. Similarly, for $f\in C_c(\CG_{X_i})$ and
$g\in \CG_{X_1\cap X_2}$, we obtain
\[
\chi_i(\gamma^{(i)}_z f)(g)
=
(\gamma^{(12)}_z\chi_i(f))(g).
\]
Thus
\[
\chi_i\circ\gamma^{(i)}_z
=
\gamma^{(12)}_z\circ\chi_i
\]
on $C_c(\CG_{X_i})$ for $i=1,2$.

Since the gauge actions and the maps \(\rho_i\) and \(\chi_i\) are continuous,
the above identities extend from the dense subalgebras of compactly
supported functions to the full groupoid \(C^*\)-algebras. Hence all maps in
the diagram are \(\mathbb T\)-equivariant.

Let
\[
\Theta:C^*(\CG)
\longrightarrow
C^*(\CG_{X_1})
\oplus_{C^*(\CG_{X_1\cap X_2})}
C^*(\CG_{X_2})
\]
be the \( * \)-isomorphism induced by the pullback diagram, namely
\[
\Theta(a)=(\rho_1(a),\rho_2(a)).
\]
We equip the pullback algebra with the componentwise gauge action
\[
\delta_z(b_1,b_2)
=
(\gamma^{(1)}_z(b_1),\gamma^{(2)}_z(b_2)).
\]
Then, for \(a\in C^*(\CG)\) and \(z\in \mathbb T\), we have
\[
\delta_z(\Theta(a))
=
(\gamma^{(1)}_z(\rho_1(a)),\gamma^{(2)}_z(\rho_2(a)))
=
(\rho_1(\gamma_z(a)),\rho_2(\gamma_z(a)))
=
\Theta(\gamma_z(a)).
\]
Thus, \(\Theta\) is \(\mathbb T\)-equivariant. Therefore the diagram is a
pullback diagram in the category of \(\mathbb T\)-\(C^*\)-algebras and
\(\mathbb T\)-equivariant \( * \)-homomorphisms.
\end{proof}

\section{Applications}

\subsection{Examples I and II: Graphs and Relative Graphs}

We begin with directed graphs and relative graphs, showing  showing how their pullback results about admissible dcomposition follow from Theorem~\ref{thm:disjoint open invariant sets}.
The key step is to translate admissible combinatorial decompositions into
decompositions of the boundary path spaces of the associated graph groupoids.
After recalling the necessary boundary path spaces and groupoids, we establish
the corresponding decomposition results and apply the groupoid pullback
theorem.

\subsubsection{Boundary path spaces}

Let $E=(E^0,E^1,r,s)$ be a directed graph, where $E^0$ and $E^1$ denote the sets of vertices and edges, respectively, and
$r,s\colon E^1\to E^0$ are the range and source maps.

A vertex $v\in E^0$ is called a \emph{sink} if $s^{-1}(v)=\emptyset$, and an
\emph{infinite emitter} if $\lvert s^{-1}(v)\rvert=\infty$.
The vertex $v$ is called \emph{regular} if $0<\lvert s^{-1}(v)\rvert<\infty$,
and \emph{singular} otherwise.

We write
\[
\operatorname{reg}(E)
:=\{v\in E^0 : 0<\lvert s^{-1}(v)\rvert<\infty\},
\qquad
\operatorname{sing}(E)
:=E^0\setminus \operatorname{reg}(E).
\]

A \emph{finite path} in $E$ is a sequence $u=e_1\cdots e_n$ of edges such that
$r(e_i)=s(e_{i+1})$ for $1\le i<n$.
We write $|u|=n$ for its length, and identify each vertex $v\in E^0$ with a path of length $0$.
For each $n\ge 0$, let $E^n$ denote the set of all paths of length $n$ in $E$, and set
\[
E^*:=\bigcup_{n\ge 0} E^n.
\]

An \emph{infinite path} in $E$ is an infinite sequence $e_1e_2\cdots$ of edges in $E$
such that $r(e_i)=s(e_{i+1})$ for all $i\in\mathbb N$.
We write $E^\infty$ for the set of all infinite paths in $E$, and define 
\emph{the path space} of $E$ by
\[
E^{\le\infty}:= E^\infty \cup E^*.
\]

A \emph{relative graph} (see \cite{MT2004} and
\cite[Definition~5.1]{BS}) is a pair $(E,R)$ consisting of a directed
graph $E$ and a subset $R\subseteq \operatorname{reg}(E)$. Following \cite[Section~2]{CL2016}, the \emph{relative boundary path space}
of $(E,R)$ is defined by
\[
\partial_R E
:=E^\infty \cup \{u\in E^* : r(u)\notin R\}.
\]
When $R=\operatorname{reg}(E)$, we set
\[
\partial E := \partial_{\operatorname{reg}(E)} E,
\]
and call it the \emph{boundary path space} of $E$.
Equivalently,
\[
\partial E
= E^\infty \;\sqcup\;
\bigl\{ u \in E^* \mid r(u) \in \operatorname{sing}(E) \bigr\}.
\]
Let $u=e_1\cdots e_n\in E^n$ and fix $0\le m\le n$.
We denote by $u(0,m)$ the initial subpath of $u$ of length $m$, namely
$u(0,m)=e_1\cdots e_m$ for $1\le m\le n$, and $u(0,0)=s(u)$.
Similarly, for an infinite path $x=e_1e_2\cdots\in E^\infty$ and $m\ge0$, we define
$x(0,m)=e_1\cdots e_m$ for $m\ge1$, and $x(0,0)=s(x)$.

If $u=e_1\cdots e_n\in E^n$ and $x=e_1e_2\cdots\in E^\infty$ satisfy
$r(u)=s(x)$, then we write $ux$ for the infinite path obtained by concatenating
$u$ and $x$.

For $u\in E^*$ and $x\in E^{\le\infty}$, we write $u\le x$ if
$|u|\le |x|$ and $x(0,|u|)=u$.
Also, we write $u<x$ if $u\le x$ and $u\neq x$.
For $u\in E^*$, we define the \emph{cylinder set} by
\[
Z(u):=\{x\in E^{\le\infty} : u\le x\}.
\]
Let $\mathcal{F}(E^*)$ be the family of all finite subsets of $E^*$.
For $u\in E^*$ and $F\in \mathcal{F}(E^*)$, we set
\[
Z_F(u):=Z(u)\setminus \bigcup_{\substack{u'\in F\\ u\le u'}} Z(u').
\]
We endow the path space $E^{\le\infty}$ with the topology generated by
\[
\{Z(u) : u\in E^*\}\,\cup\,\{E^{\le\infty}\setminus Z(u) : u\in E^*\}.
\]
By \cite[Proposition~2.2]{CL2016}, with this topology,
$E^{\le\infty}$ is a totally disconnected locally compact Hausdorff
space. Moreover, the collection
\[
\{Z_F(u) : u\in E^*,\ F\in\mathcal{F}(E^*)\}
\]
forms a basis of compact open subsets for the topology of
$E^{\le\infty}$.

Let \(R\subseteq\operatorname{reg}(E)\). We equip the relative boundary path
space \(\partial_R E\) with the subspace topology inherited from
\(E^{\leq\infty}\). Then \(\partial_R E\) is also a totally disconnected
locally compact Hausdorff space by \cite[Corollary~2.3]{CL2016}.

We will also use the fact that if \(F\) is a subgraph of \(E\), then
\(\partial F\) is closed in \(\partial E\) (see \cite[Lemma~3.3]{BS}).

\subsubsection{Boundary and relative boundary path groupoids}
Let $E$ be a directed graph.
Following \cite{BS}, we begin by recalling the groupoid $G(E)$ associated to $E$,
whose unit space is the full path space
\[
G(E)^{(0)}=E^{\le\infty}.
\]
Throughout this paper, however, we adopt the convention that paths are
concatenated from left to right; accordingly, the roles of the source and range
maps are interchanged compared with the convention used in \cite{BS}.

We define
\[
E^* * E^* * G(E)^{(0)}
:=\{(\alpha,\beta,x)\in E^*\times E^*\times G(E)^{(0)}
: r(\alpha)=r(\beta)=s(x)\}.
\]
The groupoid associated to $E$ is defined as
\[
G(E):=\bigl(E^* * E^* * G(E)^{(0)}\bigr)\big/\!\sim,
\]
where $(\alpha,\beta,x)\sim(\alpha',\beta',x')$ if there exist
$\gamma,\gamma'\in E^*$ and $y\in G(E)^{(0)}$ such that
\[
x=\gamma y,\qquad x'=\gamma' y,\qquad
\alpha\gamma=\alpha'\gamma',\qquad
\beta\gamma=\beta'\gamma'.
\]
We denote the equivalence class of $(\alpha,\beta,x)$ by $[\alpha,\beta,x]$.
The inversion is given by
$
[\alpha,\beta,x]^{-1}=[\beta,\alpha,x].
$
For composition, suppose that $\beta x=\alpha' x'$.
It can be shown that there exist $y\in G(E)^{(0)}$ and
$\gamma,\gamma'\in E^*$ such that
\[
x=\gamma y,\qquad x'=\gamma' y,\qquad \beta\gamma=\alpha'\gamma'.
\]
In this case, the product is defined by
\[
[\alpha,\beta,x]\cdot[\alpha',\beta',x']
=[\alpha\gamma,\beta'\gamma',y].
\]
The range and source maps are given by
\[
r([\alpha,\beta,x])=[r(\alpha),r(\alpha),\alpha x],
\qquad
s([\alpha,\beta,x])=[r(\beta),r(\beta),\beta x].
\]
Identifying $x\in G(E)^{(0)}$ with $[r(x),r(x),x]\in G(E)$, we may 
write
\[
r([\alpha,\beta,x])=\alpha x,
\qquad
s([\alpha,\beta,x])=\beta x.
\]

For $(\alpha,\beta)\in E^*\times E^*$ and a subset
$F\subseteq Z(r(\alpha))$, 
we define
\[
[\alpha,\beta,F]:=\{[\alpha,\beta,x]:x\in F\}.
\]
If $F$ is compact open in $G(E)^{(0)}$, then $[\alpha,\beta,F]$ is a compact open
bisection.
Such sets form a basis for a topology on $G(E)$, with respect to which $G(E)$
is an ample Hausdorff groupoid.

Let \(R\subseteq \operatorname{reg}(E)\). Since
\[
\partial_R E
=
G(E)^{(0)}\setminus E^*R,
\qquad
E^*R
:=
\{u\in E^*:r(u)\in R\},
\]
the relative boundary path space \(\partial_R E\) is a closed invariant
subset of \(G(E)^{(0)}\). We therefore define the \emph{relative boundary
path groupoid} of \((E,R)\) by
\[
G(E,R)
:=
G(E)_{\partial_R E}.
\]
It is an ample Hausdorff groupoid with unit space
$G(E,R)^{(0)}
=
\partial_R E$. 
In the special case \(R=\operatorname{reg}(E)\), we write
\[
\CG(E)
:=
G(E,\operatorname{reg}(E))
\]
and call it the \emph{boundary path groupoid} of \(E\).

There is a canonical continuous cocycle
\[
c_E:G(E)\longrightarrow\mathbb Z,
\qquad
c_E([\alpha,\beta,x])
=
|\alpha|-|\beta|.
\]
Since \(\partial_R E\) is invariant, the restriction of \(c_E\) to
\(G(E,R)\) is again a continuous \(\mathbb Z\)-valued cocycle and induces
the gauge action on \(C^*(G(E,R))\). We use the same notation \(c_E\) for
this restricted cocycle and, in particular, for its restriction to
\(\CG(E)\).

Finally, we notice that \(C^*(G(E,R))\cong C^*(E,R)\) in the category of of $\mathbb T$-$C^*$-algebras.

\subsubsection{Example I: Admissible Decompositions of Graphs}

We first discuss admissible decompositions of directed graphs and the
corresponding decompositions of boundary path spaces. Although the
directed graph case could be obtained as a consequence of the relative graph
or topological graph cases, we treat it separately here.  Directed graphs
provide the most familiar setting, and the role of admissible decompositions
is particularly transparent in this context.  This also allows us to formulate
the boundary-path decomposition directly, without passing through a more
general framework.

We  give an adapted formulation of admissible decompositions of graphs.
The terminology comes from \cite[Definition~2.1]{HRT}, but the conditions are
presented in a form that parallels the relative graph and topological graph
settings considered later.  In particular, condition {\rm (d)} should be
viewed as the directed graph analogue of the corresponding condition in those
more general settings.
In this form, admissibility is precisely the graph-theoretic counterpart of the equivalent conditions appearing in the pullback theorem below.

\begin{dfn}A pair $\{F_1, F_2\}$ of subgraphs of $E$ is called an {\it admissible decomposition} (cf. \cite[Definition~2.1]{HRT}) of a graph $E$ if
\begin{enumerate}
\item[(a)] $E = F_1 \cup F_2$,
\item[(b)] $\operatorname{sink}(F_1 \cap F_2) \subseteq \operatorname{sing}(F_1) \cap \operatorname{sing}(F_2)$,
\item[(c)] $F_1^1 \cap F_2^1 = r_{F_1}^{-1}(F_1^0 \cap F_2^0) = r_{F_2}^{-1}(F_1^0 \cap F_2^0)$,
\item[(d)] $\operatorname{sing}(F_1) \cap 
\operatorname{reg}(F_1\cap F_2)
\subseteq
\operatorname{reg}(F_2).$
\end{enumerate}
\end{dfn}

Let $E$ be a directed graph.  A subset \(H\subseteq E^0\) is called \emph{hereditary} if
$$
s(e)\in H \implies r(e)\in H
$$
for every \(e\in E^1\). It is called \emph{saturated} if whenever
\(v\in \operatorname{reg}(E) \) and every edge \(e\in E^1\) with \(s(e)=v\)
satisfies
$
r(e)\in H,
$
then \(v\in H\).

\begin{remark} \label{properties of admissible decomposition}
Let $\{F_1,F_2\}$ be an admissible decomposition of a graph $E$.
For $\{i,j\}=\{1,2\}$, put
\[
H_i
:=
F_i^0\setminus(F_1^0\cap F_2^0)
=
F_i^0\setminus F_j^0
=
E^0\setminus F_j^0.
\]
Then the following observations hold.
\begin{enumerate}
\item[(i)]
Put $F_0:=F_1^0\cap F_2^0$. 
By \cite[Lemma~2.2]{HRT} and \cite[Example~7.3]{BS},
condition {\rm (c)} is equivalent to saying that $H_i$ is hereditary
in $F_i$, and hence in $E$, for $i=1,2$, and that
\[
F_1^1\cap F_2^1=F_0E^1F_0,
\]
where $F_0E^1F_0
:=
\{e\in E^1:s_E(e),r_E(e)\in F_0\}$. In this case,
\begin{align*}
F_1 &= E/H_2, \\
F_2 &= E/H_1, \\
F_1\cap F_2 &= F_i/H_i
\qquad (i=1,2).
\end{align*}

\item[(ii)]
By \cite[Lemma~2.3]{HRT} and \cite[Example~7.3]{BS},
condition {\rm (b)} is equivalent to $H_i$ being saturated in both
$F_i$ and $E$ for $i=1,2$.

\item[(iii)]
By \cite[Example~7.2]{BS}, condition {\rm (d)} admits the following
equivalent formulations:
\[
\begin{aligned}
&\operatorname{sing}(F_1)\cap
\operatorname{reg}(F_1\cap F_2)
\subseteq
\operatorname{reg}(F_2)
\\
&\quad\Longleftrightarrow\quad
\operatorname{sing}(F_2)\cap
\operatorname{reg}(F_1\cap F_2)
\subseteq
\operatorname{reg}(F_1)
\\
&\quad\Longleftrightarrow\quad
\operatorname{sing}(F_1)\cap
\operatorname{sing}(F_2)\cap
\operatorname{reg}(F_1\cap F_2)
=
\emptyset.
\end{aligned}
\]
In particular, condition {\rm (d)} is symmetric in $F_1$ and $F_2$.
\end{enumerate}
\end{remark}

We now show that the diagram~\eqref{diag:Cstar-decomp} 
is a pullback by applying Theorem~\ref{thm:disjoint open invariant sets}. 
To this end, we first prove the following result 
concerning the boundary paths of \(E\), 
which follows from the structural properties described in the preceding remark.

\begin{prop}\label{decomposition of the boundary path space:graph}
Let \(E = (E^0, E^1, r, s)\) be a directed graph, and let \(\{F_1, F_2\}\) be an
admissible decomposition of \(E\). Then we have 
\begin{enumerate}
\item  $ \partial E (=\partial (F_1 \cup  F_2))=\partial F_1 \cup \partial F_2 $.
\item $\partial (F_1 \cap F_2) =\partial F_1 \cap \partial F_2 $.
\end{enumerate}
\end{prop}

\begin{proof}
\noindent\text{(1) ($\subseteq$):}
We begin with the following observation: let
$
x=x_1x_2\cdots
$
be a path in $E$. Suppose that there exists $n \ge 0$ such that 
$r(x_n) \in F_1^0 \setminus F_2^0$, and that $n$ is minimal with this property.
Then every edge of $x$ belongs to $F_1^1$;
indeed, since $F_1^0 \setminus F_2^0$ is hereditary, we have 
$r(x_k) \in F_1^0 \setminus F_2^0$ for all $k \ge n$, and hence 
$x_k \in F_1^1$ for all $k \ge n$.
For $k < n$, minimality of $n$ implies that 
$r(x_k) \in F_1^0 \cap F_2^0$, and hence by admissible condition~(c), 
$x_k \in F_1^1 \cap F_2^1 \subseteq F_1^1$.
Thus every edge of $x$ belongs to $F_1^1$.
 
By symmetry, if $x$ eventually enters $F_2^0 \setminus F_1^0$, then every edge of $x$ belongs to $F_2^1$.

Now, take any boundary path $x \in \partial E$. We consider three possibilities
according to the vertices visited by $x$.

\medskip
\noindent \text{Case 1: $x$ eventually enters a vertex in $F_1^0 \setminus F_2^0$.}

By the observation above, every edge of $x$ belongs to $F_1^1$, and hence $x$ is a path in $F_1$.
If $x$ is infinite, then  $x \in \partial F_1$.
If $x$ is finite and terminates at a vertex $v \in \operatorname{sing}(E)$, 
then $v$ is also singular in $F_1$, and hence $x \in \partial F_1$.

\medskip
\noindent\text{Case 2: $x$ eventually enters a vertex in $F_2^0 \setminus F_1^0$.}

By the same argument as in Case~1, with the roles of $F_1$ and $F_2$ interchanged,
we conclude that $x \in \partial F_2$.

\medskip
\noindent\text{Case 3: $x$ never enters $F_1^0 \setminus F_2^0$ nor $F_2^0 \setminus F_1^0$.}

In this case, for every edge $x_n$ of $x$, we have
\[
r(x_n) \in F_1^0 \cap F_2^0.
\]
By admissible condition~(c) $
F_1^1 \cap F_2^1
=
r_{F_1}^{-1}(F_1^0 \cap F_2^0)
=
r_{F_2}^{-1}(F_1^0 \cap F_2^0),
$
 every edge of $x$ belongs to $F_1^1 \cap F_2^1$.
Therefore, $x$ is a path in both $F_1$ and $F_2$. Now, if $x$ is infinite, then $x$ is clearly an infinite path in both $F_1$ and $F_2$,
and hence $x \in \partial F_1 \cap \partial F_2$.

Suppose next that $x$ is finite and terminates at a vertex $v \in \operatorname{sing}(E)$. If $v$ is a sink in $E$, then $v$ is a sink both in $F_1$ and $F_2$, and hence $x\in \partial F_1 \cap \partial F_2\subseteq \partial F_1 \cup \partial F_2$. If $v$ is an infinite emitter in $E$, because $E=F_1\cup F_2$, $v$ is an infinite emitter in $F_1$ or in $F_2$, from where it follows that $x\in \partial F_1 \cup \partial F_2$.

\medskip
Combining the three cases above, we obtain
$$
\partial E \subseteq \partial F_1 \cup \partial F_2.
$$

\noindent\text{ ($\supseteq$):} Since $H_2 = E^0 \setminus F_1^0$ is saturated, we have
\[
\operatorname{reg}(E) \cap F_1^0 \subseteq \operatorname{reg}(F_1)
\]
(see, for example, \cite[Example 7.2]{BS}). Similarly, we have
\[
\operatorname{reg}(E) \cap F_2^0 \subseteq \operatorname{reg}(F_2).
\]
It follows from \cite[Theorem 3.4]{BS} that
\[
\partial F_1 \subseteq \partial E
\quad \text{and} \quad
\partial F_2 \subseteq \partial E.
\]
Therefore,
\[
\partial F_1 \cup \partial F_2 \subseteq \partial E.
\]

\noindent\text{(2) ($\subseteq$):} Let $x \in \partial(F_1 \cap F_2)$.  
 If $x$ is an infinite path, then $x \in F_1^\infty \cap F_2^\infty$, so $x \in \partial F_1 \cap \partial F_2$. 
Suppose that $x$ is a finite path ending at $v := r(x) \in \operatorname{sing}(F_1\cap F_2)$. 
By admissible condition~(b), every sink of $F_1 \cap F_2$ is singular in each $F_i$. 
Moreover, every infinite emitter of $F_1 \cap F_2$ is an infinite emitter in each $F_i$. 
Thus, $v$ is singular in both $F_1$ and $F_2$, and hence $x \in \partial F_1 \cap \partial F_2$.

Therefore, we have $$
\partial(F_1 \cap F_2) \subseteq \partial F_1 \cap \partial F_2 .
$$

 \noindent \text{($\supseteq$):}
Let $x\in \partial F_1\cap \partial F_2.$
Since $x$ is a path in both $F_1$ and $F_2$, every edge of $x$ belongs to
$F_1^1\cap F_2^1=(F_1\cap F_2)^1.$
Hence, $x$ is a path in $F_1\cap F_2$.
If $x$ is infinite, then
\[
x\in (F_1\cap F_2)^\infty\subseteq \partial(F_1\cap F_2).
\]
Suppose now that $x$ is finite, and put $v=r(x).$
Since $x\in \partial F_1\cap \partial F_2,$
we have
\[
v\in \operatorname{sing}(F_1)\cap \operatorname{sing}(F_2).
\]
 By Remark~\ref{properties of admissible decomposition}(iii),
\[
\operatorname{sing}(F_1)\cap \operatorname{sing}(F_2)
\cap \operatorname{reg}(F_1\cap F_2)
=
\emptyset,\]
which implies that $v\in \operatorname{sing}(F_1\cap F_2)$. Thus, $x\in \partial(F_1\cap F_2).$ 

Consequently,
$\partial F_1\cap \partial F_2
\subseteq
\partial(F_1\cap F_2).
$

\end{proof}

The following result recovers the pullback theorem for admissible
decompositions of graphs from \cite[Theorem~2.7]{HRT} as an application of
Theorem~\ref{thm:disjoint open invariant sets}.

\begin{thm}[Pullbacks from admissible decompositions of graphs]
\label{ex 1:graph pullback theorem}
Let \(E=(E^0,E^1,r,s)\) be a directed graph, and let
\(\{F_1,F_2\}\) be an admissible decomposition of \(E\). Then the diagram
\begin{equation}\label{diag:Cstar-decomp}
\begin{aligned}
\begin{tikzpicture}
  \node (P) at (0,2) {$C^*(E)$};
  \node (A) at (-2,0) {$C^*(F_1)$};
  \node (B) at (2,0) {$C^*(F_2)$};
  \node (C) at (0,-2) {$C^*(F_1\cap F_2)$};

  \draw[->] (P) -- (A);
  \draw[->] (P) -- (B);
  \draw[->] (A) -- (C);
  \draw[->] (B) -- (C);
\end{tikzpicture}
\end{aligned}
\end{equation}
is a pullback diagram in the category of \(\mathbb T\)-\(C^*\)-algebras
and \(\mathbb T\)-equivariant \( * \)\nobreakdash-ho\-mo\-mor\-phisms, where each graph
\(C^*\)-algebra is equipped with its canonical gauge action.
\end{thm}

\begin{proof}
Consider the boundary path groupoid \(\mathcal{G}(E)\) of \(E\) with unit
space \(\partial E\). By Proposition~\ref{decomposition of the boundary path space:graph}(1), we have
$
\partial E=\partial F_1\cup \partial F_2.
$
Since \(\partial F_1\) and \(\partial F_2\) are closed
\(\mathcal{G}(E)\)-invariant subsets of \(\partial E\), the sets
\[
U_1:=\partial E\setminus \partial F_1,
\qquad
U_2:=\partial E\setminus \partial F_2
\]
are open \(\mathcal{G}(E)\)-invariant subsets of \(\partial E\). Moreover,
the equality \(\partial E=\partial F_1\cup \partial F_2\) implies
$U_1\cap U_2=\emptyset$. Also, we have 
$
\mathcal{G}(E)
=
\mathcal{G}(E)_{\partial F_1}
\cup
\mathcal{G}(E)_{\partial F_2}$.
Moreover, by Proposition~\ref{decomposition of the boundary path space:graph}(2), we see that 
\[
\mathcal{G}(E)_{\partial F_1}
\cap
\mathcal{G}(E)_{\partial F_2}
=
\mathcal{G}(E)_{\partial F_1\cap \partial F_2}
=
\mathcal{G}(E)_{\partial(F_1\cap F_2)}.
\]

Let
$
c_E:\mathcal{G}(E)\to \mathbb Z
$
be the canonical cocycle, given by
\[
c_E([\alpha,\beta,x])=|\alpha|-|\beta|.
\]
Its restrictions to
\(\mathcal{G}(E)_{\partial F_i}\), for \(i=1,2\), and to
\(\mathcal{G}(E)_{\partial(F_1\cap F_2)}\) are precisely the canonical
cocycles on the corresponding graph groupoids.  Therefore, by
Theorem~\ref{thm:disjoint open invariant sets}, we obtain the following pullback diagram
in the category of \(\mathbb T\)-\(C^*\)-algebras and
\(\mathbb T\)-equivariant \( * \)-homomorphisms:
\[
\begin{tikzpicture}
  \node (P) at (0,2) {$C^*(\mathcal{G}(E))$};
  \node (A) at (-2,0) {$C^*(\mathcal{G}(E)_{\partial F_1})$};
  \node (B) at (2,0) {$C^*(\mathcal{G}(E)_{\partial F_2})$};
  \node (C) at (0,-2) {$C^*(\mathcal{G}(E)_{\partial(F_1\cap F_2)})$.};

  \draw[->] (P) -- (A);
  \draw[->] (P) -- (B);
  \draw[->] (A) -- (C);
  \draw[->] (B) -- (C);
\end{tikzpicture}
\]
Here all gauge actions are those induced by the canonical cocycle \(c_E\)
and its restrictions.

Now observe that
\[
\mathcal{G}(E)_{\partial F_i}
=
\mathcal{G}(F_i)
\quad \text{for } i=1,2,
\qquad
\mathcal{G}(E)_{\partial(F_1\cap F_2)}
=
\mathcal{G}(F_1\cap F_2).
\]
Under these identifications, the restricted cocycle \(c_E\) agrees with the
canonical cocycle on each graph groupoid. Hence the standard isomorphisms
\[
C^*(\mathcal{G}(E))\cong C^*(E),
\qquad
C^*(\mathcal{G}(E)_{\partial F_i})\cong C^*(F_i)
\quad \text{for } i=1,2,
\]
and
\[
C^*(\mathcal{G}(E)_{\partial(F_1\cap F_2)})
\cong
C^*(F_1\cap F_2)
\]
are \(\mathbb T\)-equivariant with respect to the gauge actions.

Therefore, after applying these \(\mathbb T\)-equivariant isomorphisms to the
above diagram, we obtain the diagram
\[
\begin{tikzpicture}
  \node (P) at (0,2) {$C^*(E)$};
  \node (A) at (-2,0) {$C^*(F_1)$};
  \node (B) at (2,0) {$C^*(F_2)$};
  \node (C) at (0,-2) {$C^*(F_1\cap F_2)$};

  \draw[->] (P) -- (A);
  \draw[->] (P) -- (B);
  \draw[->] (A) -- (C);
  \draw[->] (B) -- (C);
\end{tikzpicture}
\]
as a pullback diagram.  This pullback is taken in the category of
\(\mathbb T\)-\(C^*\)-algebras and \(\mathbb T\)-equivariant
$*$-homomorphisms.
\end{proof}

\begin{remark}
By the relative graph pullback characterization \cite[Theorem 5.9]{BS}, the converse also holds
after viewing ordinary directed graphs as a special case of relative graphs.
Hence, in the directed graph setting, the admissibility conditions above are
equivalent to the assertion that the diagram in
Theorem~\ref{ex 1:graph pullback theorem} is a pullback diagram.  \end{remark}

\subsubsection{Example II: Admissible Pairs of Relative Graphs}

We now focus on relative graphs and explain how the pushout construction in this category yields a pullback diagram of 
$C^*$-algebra via the associated relative boundary path groupoids. This gives another class of examples of Theorem~\ref{thm:disjoint open invariant sets}, extending the result for relative graphs in \cite[Theorem 5.9]{BS}.

A \emph{morphism} (\cite[Definition 4.1]{BS}) of relative graphs
\[
\alpha:(F,B)\longrightarrow (E,A)
\]
is an injective graph homomorphism $\alpha:F\hookrightarrow E$, viewed as an inclusion, satisfying:
\begin{enumerate}
\item $E^0\setminus F^0$ is hereditary in $E$;
\item $F^1=F^0E^1F^0$;
\item $A\cap F^0\subseteq B$.
\end{enumerate}

We note that if $\alpha:(F,B)\longrightarrow (E,A)$ is a morphism, then $$G(F,B)^{(0)} \subseteq G(E,A)^{(0)}$$ by \cite[Theorem 2.4]{BS}. 

Let
\[
(F_0,A_0) \xrightarrow{\;\alpha_1\;} (F_1,A_1), 
\qquad
(F_0,A_0) \xrightarrow{\;\alpha_2\;} (F_2,A_2)
\]
be morphisms of relative graphs, and let $(E,A)$ be the pushout of this diagram
in the category of relative graphs given by
\begin{align*}
E &= F_1 \mathop{\sqcup}_{F_0} F_2, \\
A &= (A_1 \setminus F_0^0) \cup (A_2 \setminus F_0^0) \cup (A_1 \cap A_2),
\end{align*}
where the graph $E$ is defined by
\begin{align*}
E^0 &= (F_1^0 \sqcup F_2^0)/\{\alpha_1(v)=\alpha_2(v) : v\in F_0^0\}, \\
E^1 &= (F_1^1 \sqcup F_2^1)/\{\alpha_1(e)=\alpha_2(e) : e\in F_0^1\}
\end{align*}
with $r_E(e)=r_{F_i}(e)$ and $s_E(e)=s_{F_i}(e)$ for $e\in F_i^1$ (see \cite[Theorem 4.3]{BS}).

The pair $(\alpha_1, \alpha_2)$ is called {\it admissible} (\cite[Definition 5.7]{BS}) if 
$$A_0 \subseteq A_1 \cup A_2.$$
Note that $(\alpha_1, \alpha_2)$ is admissible if and only if $A_0=A_{12}:=(A_1 \cup A_2) \cap F_0^0.$

\begin{prop}\label{boundary path decomposition:relative graph}
Let $\alpha_i:(F_0,A_0)\to(F_i,A_i)$, $i=1,2$, be morphisms of relative graphs
and let $(E,A)$ be their pushout. If $(\alpha_1,\alpha_2)$ is admissible, then
we have the following:
\begin{enumerate}
\item $G(F_1,A_1)^{(0)} \cup G(F_2,A_2)^{(0)} = G(E,A)^{(0)}$,
\item  $G(F_1,A_1)^{(0)} \cap G(F_2,A_2)^{(0)}= G(F_0,A_0)^{(0)}$.
\end{enumerate}
\end{prop}

\begin{proof}
We first note that, for $i=1,2$, if $\mu\in F_i^*$ and
$r(\mu)\in F_0^0$, then $\mu\in F_0^*$. Indeed, since
$F_i^0\setminus F_0^0$ is hereditary in $F_i$, all vertices of $\mu$
belong to $F_0^0$, and then
\[
F_0^1=F_0^0F_i^1F_0^0
\]
implies that all edges of $\mu$ belong to $F_0^1$.

\medskip

\noindent
(1): By the construction of the pushout,
\[
G(E)^{(0)}
=
G(F_1)^{(0)}\cup G(F_2)^{(0)}.
\]
Also, for $i=1,2$,
\[
G(F_i)^{(0)}\setminus F_i^*A_i
=
G(F_i)^{(0)}\setminus E^*A_i,
\]
since the paths under consideration already belong to $F_i$.

We compute
\begin{align*}
G(F_1,A_1)^{(0)}\cup G(F_2,A_2)^{(0)}
&=
\bigl(G(F_1)^{(0)}\setminus F_1^*A_1\bigr)
\cup
\bigl(G(F_2)^{(0)}\setminus F_2^*A_2\bigr)\\
&=
\bigl(G(F_1)^{(0)}\setminus E^*A_1\bigr)
\cup
\bigl(G(F_2)^{(0)}\setminus E^*A_2\bigr)\\
&=
\bigl(G(F_1)^{(0)}\cup G(F_2)^{(0)}\bigr)\setminus E^*A\\
&=
G(E)^{(0)}\setminus E^*A\\
&=
G(E,A)^{(0)},
\end{align*}
where the third equality is justified by the following two inclusions. Since
$F_1^0\cap F_2^0=F_0^0$, the definition of $A$ gives
\[
A\cap F_1^0
=
(A_1\setminus F_0^0)\cup(A_1\cap A_2)
\subseteq A_1,
\]
and similarly,
\[
A\cap F_2^0\subseteq A_2.
\]

Now let
\[
x\in
\bigl(G(F_1)^{(0)}\setminus E^*A_1\bigr)
\cup
\bigl(G(F_2)^{(0)}\setminus E^*A_2\bigr).
\]
Suppose, for instance, that
$x\in G(F_1)^{(0)}\setminus E^*A_1$.
If $x$ is infinite, then $x\notin E^*A$. If $x$ is finite, then
$r(x)\in F_1^0$ and $r(x)\notin A_1$. Since
$A\cap F_1^0\subseteq A_1$, it follows that $r(x)\notin A$, and hence
$x\notin E^*A$. The same argument applies when
$x\in G(F_2)^{(0)}\setminus E^*A_2$. Thus
\[
\bigl(G(F_1)^{(0)}\setminus E^*A_1\bigr)
\cup
\bigl(G(F_2)^{(0)}\setminus E^*A_2\bigr)
\subseteq
\bigl(G(F_1)^{(0)}\cup G(F_2)^{(0)}\bigr)\setminus E^*A.
\]

For the reverse inclusion, let
\[
x\in
\bigl(G(F_1)^{(0)}\cup G(F_2)^{(0)}\bigr)\setminus E^*A.
\]
If $x$ is infinite, then $x\notin E^*A_1\cup E^*A_2$, so the conclusion
is immediate. Suppose that $x$ is finite and, without loss of
generality, $x\in G(F_1)^{(0)}$.

If $r(x)\notin A_1$, then
\[
x\in G(F_1)^{(0)}\setminus E^*A_1,
\]
and we are done. Otherwise, $r(x)\in A_1$. Since $r(x)\notin A$, while
\[
A_1\setminus F_0^0\subseteq A
\qquad\text{and}\qquad
A_1\cap A_2\subseteq A,
\]
we must have
\[
r(x)\in (A_1\cap F_0^0)\setminus A_2.
\]
By the observation above,
\[
x\in G(F_0)^{(0)}\subseteq G(F_2)^{(0)},
\]
and since $r(x)\notin A_2$,
\[
x\in G(F_2)^{(0)}\setminus E^*A_2.
\]
Thus the reverse inclusion follows, proving the third equality.

\medskip

(2) By the construction of the pushout,
\[
G(F_1)^{(0)}\cap G(F_2)^{(0)}
=
G(F_0)^{(0)}.
\]

We compute
\begin{align*}
G(F_1,A_1)^{(0)} \cap G(F_2,A_2)^{(0)}
&=
\bigl(G(F_1)^{(0)} \setminus F_1^*A_1\bigr)
\cap
\bigl(G(F_2)^{(0)} \setminus F_2^*A_2\bigr)\\
&=
\bigl(G(F_1)^{(0)} \cap G(F_2)^{(0)}\bigr)
\setminus E^*A_0\\
&=
G(F_0)^{(0)}\setminus E^*A_0\\
&=
G(F_0,A_0)^{(0)},
\end{align*}
where the second equality is justified as follows. Let
$
x\in G(F_0)^{(0)}.
$
We claim that $x\in \bigl(G(F_1)^{(0)} \setminus F_1^*A_1\bigr)
\cap
\bigl(G(F_2)^{(0)} \setminus F_2^*A_2\bigr)$ if and only if $\bigl(G(F_1)^{(0)} \cap G(F_2)^{(0)}\bigr)
\setminus E^*A_0$.
If $x$ is infinite, then the claim is immediate, since
$
x\notin F_1^*A_1,$ $x\notin F_2^*A_2$ and $
x\notin E^*A_0.$

Suppose that $x$ is finite. Then $r(x)\in F_0^0$, and since
$(\alpha_1,\alpha_2)$ is admissible,
\[
A_0=(A_1\cup A_2)\cap F_0^0.
\]
Hence, we see that
\begin{align*}
x\notin F_1^*A_1 \text{ and } x\notin F_2^*A_2
&\Longleftrightarrow
r(x)\notin A_1 \text{ and } r(x)\notin A_2\\
&\Longleftrightarrow
r(x)\notin A_1\cup A_2\\
&\Longleftrightarrow
r(x)\notin A_0\\
&\Longleftrightarrow
x\notin F_0^*A_0,
\end{align*}
Hence the claim also holds for $x$ finite.
This proves the second equality and we are done.
\end{proof}

Note that \cite[Theorem~5.9]{BS} gives an equivalence between admissibility of
the pair \((\alpha_1,\alpha_2)\) and the corresponding pullback property.
The following result gives an alternative proof of the
admissibility-implies-pullback direction from a groupoid point of view,
using the groupoid pullback Theorem~\ref{thm:disjoint open invariant sets}.

\begin{thm}[Pullbacks from admissible pairs of relative graphs]
\label{thm:relative-graph-pullback}
Let $\alpha_i: (F_0, A_0) \to (F_i, A_i)$, $i=1,2$, be morphisms of relative
graphs, and let $(E,A)$ be their pushout. If $(\alpha_1, \alpha_2)$ is
admissible, then the following diagram of $C^*$-algebras
\[
\begin{tikzpicture}
  \node (P) at (0,2) {$C^*(E,A)$};
  \node (A) at (-2,0) {$C^*(F_1,A_1)$};
  \node (B) at (2,0) {$C^*(F_2,A_2)$};
  \node (C) at (0,-2) {$C^*(F_0,A_0)$};

  \draw[->] (P) -- (A);
  \draw[->] (P) -- (B);
  \draw[->] (A) -- (C);
  \draw[->] (B) -- (C);
\end{tikzpicture}
\]
is  pullback in the category of \(\mathbb T\)-\(C^*\)-algebras and
\(\mathbb T\)-equivariant \( * \)-homomorphisms.
\end{thm}

\begin{proof}
We equip the relative boundary path groupoid \(G(E,A)\) with the canonical
continuous cocycle
\[
c_{E}:G(E,A)\to \mathbb Z,
\qquad
c_{E}([\alpha,\beta,x])=|\alpha|-|\beta|.
\]
The corresponding gauge action on \(C^*(G(E,A))\) restricts to the reductions
appearing below.

Set
\begin{align*}
U^{(1)} &= G(E,A)^{(0)} \setminus G(F_1, A_1)^{(0)},\\
U^{(2)} &= G(E,A)^{(0)} \setminus G(F_2, A_2)^{(0)}.
\end{align*}
Then $U^{(1)}$ and $U^{(2)}$ are open $G(E,A)$-invariant subsets
(see \cite[Section 2]{BS}) such that
\(U^{(1)} \cap U^{(2)} = \emptyset\)
(see the proof of \cite[Proposition~5.4]{BS}).

Set
\[
X_i := G(E,A)^{(0)} \setminus U^{(i)}
= G(F_i,A_i)^{(0)},
\qquad i=1,2.
\]
By Proposition~\ref{boundary path decomposition:relative graph}, we have
\[
G(E,A)^{(0)}
=
G(F_1,A_1)^{(0)} \cup G(F_2,A_2)^{(0)},
\]
and
\[
G(F_1,A_1)^{(0)} \cap G(F_2,A_2)^{(0)}
=
G(F_0,A_0)^{(0)}.
\]
Thus, we have that
\[
G(E,A)_{X_i}=G(F_i,A_i),
\qquad i=1,2,
\]
and
\[
G(E,A)_{X_1\cap X_2}=G(F_0,A_0).
\]
Applying Theorem~\ref{thm:disjoint open invariant sets} to the disjoint open
invariant subsets \(U^{(1)}\) and \(U^{(2)}\), with the cocycle \(c_{E}\), we
obtain the following pullback diagram in the category of
\(\mathbb T\)-\(C^*\)-algebras and
\(\mathbb T\)-equivariant \( * \)-homomorphisms:
\[
\begin{tikzpicture}
  \node (P) at (0,2) {$C^*(G(E,A))$};
  \node (A) at (-2,0) {$C^*(G(F_1,A_1))$};
  \node (B) at (2,0) {$C^*(G(F_2,A_2))$};
  \node (C) at (0,-2) {$C^*(G(F_0,A_0))$};

  \draw[->] (P) -- (A);
  \draw[->] (P) -- (B);
  \draw[->] (A) -- (C);
  \draw[->] (B) -- (C);
\end{tikzpicture}
\]
Here the gauge actions on the lower three groupoid \(C^*\)-algebras are induced
by the restrictions of \(c_{E}\), which correspond to the canonical cocycles
on \(G(F_i,A_i)\), \(i=1,2\), and \(G(F_0,A_0)\).

Under the canonical gauge-equivariant isomorphisms
\[
C^*(G(E,A))\cong C^*(E,A),
\qquad
C^*(G(F_i,A_i))\cong C^*(F_i,A_i)
\quad (i=1,2),
\]
and
\[
C^*(G(F_0,A_0))\cong C^*(F_0,A_0),
\]
the above pullback diagram is identified with the desired pullback diagram
in the category of \(\mathbb T\)-\(C^*\)-algebras and
\(\mathbb T\)-equivariant \( * \)-homomorphisms.
\end{proof}

\subsection{Example III: Topological Graphs}

A quadruple \(E=(E^0,E^1,s_E,r_E)\) is called a \emph{topological graph}
if \(E^0\) and \(E^1\) are locally compact Hausdorff spaces,
\(s_E,r_E:E^1\to E^0\) are continuous maps, and \(r_E\) is a local
homeomorphism. Here we follow the source--range convention of \cite{GQT}.

Given a topological graph \(E=(E^0,E^1,s_E,r_E)\), we define the infinite
path space by
\[
E^\infty
:=
\left\{
e_1e_2\cdots \in (E^1)^{\mathbb N}
\mid
r_E(e_k)=s_E(e_{k+1})
\text{ for all } k\ge 1
\right\},
\]
equipped with the subspace topology inherited from the product topology on
\((E^1)^{\mathbb N}\).

We define the following distinguished subsets of the vertex space \(E^0\).
First, set
\[
E^0_{\mathrm{fin}}
:=
\{v\in E^0 \mid \text{there exists a neighbourhood } V \text{ of } v
\text{ such that } s_E^{-1}(V) \text{ is compact}\}.
\]
The set of \emph{infinite emitters} is defined by
\[
E^0_{\mathrm{inf}}
:=
E^0\setminus E^0_{\mathrm{fin}}.
\]
The set of \emph{sink vertices} is given by
\[
E^0_{\mathrm{sink}}
:=
E^0\setminus \overline{s_E(E^1)}.
\]
The set of \emph{regular vertices} is defined by
\[
E^0_{\mathrm{reg}}
:=
E^0_{\mathrm{fin}}\setminus \overline{E^0_{\mathrm{sink}}},
\]
and the set of \emph{singular vertices} is
\[
E^0_{\mathrm{sing}}
:=
E^0\setminus E^0_{\mathrm{reg}}.
\]

For each \(n\ge 1\), we define the set of finite paths of length \(n\) by
\[
E^n
:=
\left\{
e_1\cdots e_n \in (E^1)^n \mid
r_E(e_k)=s_E(e_{k+1})
\text{ for } 1\le k\le n-1
\right\}.
\]

We also set
\[
E^*
:=
\bigsqcup_{n=0}^{\infty}E^n,
\]
where \(E^0\) is identified with the set of paths of length zero. For a
finite path \(\alpha=(e_1,\dots,e_n)\in E^n\), we write
\[
r_E(\alpha)
:=
r_E(e_n).
\]
For \(v\in E^0\), regarded as a path of length zero, we put
\[
r_E(v):=v.
\]
The set of finite paths ending at singular vertices is defined by
\[
E^*_{\mathrm{sing}}
:=
\{\alpha\in E^* \mid r_E(\alpha)\in E^0_{\mathrm{sing}}\}.
\]

The \emph{boundary path space} (cf, \cite[Definition 3.1]{KL2017}) of \(E\) is defined by
\[
\partial E
:=
E^\infty \cup E^*_{\mathrm{sing}}.
\]
We equip \(\partial E\) with the Yeend's topology obtained by adapting
\cite[Definition~3.7]{KL2017} to our source--range convention. More precisely, for
\(S\subseteq E^*\), set
\[
Z_E(S)
:=
\{\mu\in \partial E:
\text{ either } s_E(\mu)\in S,\text{ or there exists }1\le i\le |\mu|
\text{ such that } \mu_1\cdots \mu_i\in S\}.
\]
The topology on \(\partial E\) is generated by the sets
\[
Z_E(U)\cap Z_E(K)^c,
\]
where \(U\subseteq E^*\) is open and \(K\subseteq E^*\) is compact. With this
topology, \(\partial E\) is a locally compact Hausdorff space (see \cite[Definition 3.7]{KL2017} or \cite[Proposition 3.6]{KL2017}).

We end this section with a lemma about convergence of nets on $\partial E$.

\begin{lem}\label{convergence in topological graphs}
    Let $E$ be a topological graph and $(x^{(\lambda)})_{\lambda}$ a net in $\partial E$ converging to $x$. Then
    \begin{enumerate}
        \item $(s(x^{(\lambda)}))_{\lambda}$ converges to $s(x)$ in $E^0$.
        \item For every $1\leq k\leq |x|$ with $k\neq \infty$, there exists $\lambda_0$ such that $|x^{(\lambda)}|\geq k$ and $(x_k^{(\lambda)})_{\lambda\geq\lambda_0}$ converges to $x_k\in E^1$.
    \end{enumerate} 
\end{lem}

\begin{proof}
    This follows from the description of convergence of nets in $\partial E$. This is done in \cite[Proposition~3.12]{Yeend2006} and \cite[Lemma~4.8]{KL2017} for sequences, and, more generally, in \cite[Theorem~3.10 and Section~5]{Cas2021} for nets.
\end{proof}

\subsubsection{The Deaconu--Renault groupoid}

Let \(E\) be a topological graph. We denote by
\[
\sigma_E:\partial E\setminus E^0_{\mathrm{sing}}\to \partial E
\]
the one-sided shift map. More precisely, if \(x=e_1e_2\cdots\in E^\infty\),
then
\[
\sigma_E(x)=e_2e_3\cdots,
\]
and if \(x=e_1\cdots e_n\in E^*\) with \(n\geq 1\) and
\(r_E(x)\in E^0_{\mathrm{sing}}\), then
\[
\sigma_E(x)=e_2\cdots e_n,
\]
where, when \(n=1\), the latter is understood as the vertex \(r_E(e_1)\).
By \cite[Lemma~6.1]{KL2017}, adapted to our source--range convention,
\(\sigma_E\) is a partial local homeomorphism on \(\partial E\), with
\[
\operatorname{dom}(\sigma_E)=\partial E\setminus E^0_{\mathrm{sing}}.
\]

We denote by
$
\Gamma(\partial E,\sigma_E)
$
the Deaconu--Renault groupoid associated to this partial local homeomorphism.
Thus
\[
\begin{aligned}
\Gamma(\partial E,\sigma_E)
={}&
\{(x,m-n,y)\in \partial E\times \mathbb Z\times \partial E:
m,n\in\mathbb N,\\
&\quad
x\in \operatorname{dom}(\sigma_E^m),\
y\in \operatorname{dom}(\sigma_E^n),\
\sigma_E^m(x)=\sigma_E^n(y)\}.
\end{aligned}
\]
The unit space is identified with \(\partial E\) via
\[
x\mapsto (x,0,x).
\]
The range and source maps are given by
\[
r(x,k,y)=x,
\qquad
s(x,k,y)=y,
\]
and the multiplication and inverse are given by
\[
(x,k,y)(y,l,z)=(x,k+l,z),
\qquad
(x,k,y)^{-1}=(y,-k,x).
\]

There is a canonical continuous cocycle
\[
c_E:\Gamma(\partial E,\sigma_E)\to \mathbb Z
\]
defined by
\[
c_E(x,k,y)=k.
\]
The associated gauge action on \(C^*(\Gamma(\partial E,\sigma_E))\) is the
one induced by this cocycle.

\subsubsection{Adjunctions and Union Graphs}

Let $F$ be a topological graph. 
A topological graph $G$ is called a \emph{subgraph} of $F$, denoted $G \subseteq F$, 
if $G^0 \subseteq F^0$ and $G^1 \subseteq F^1$ are subspaces and the source and range maps of $G$ are given by the restrictions of those of $F$, that is,
$
s_G = s_F|_{G^1}$ and $r_G = r_F|_{G^1}$.

A subset \(G^0\subseteq F^0\) is said to be \emph{positively invariant} if
\[
r_F(e)\in G^0 \implies s_F(e)\in G^0
\]
for every \(e\in F^1\). It is said to be \emph{negatively invariant} if, for
every \(v\in G^0\cap F^0_{\mathrm{reg}}\), there exists \(e\in F^1\) such that
\[
s_F(e)=v
\quad\text{and}\quad
r_F(e)\in G^0.
\]
We say that \(G^0\) is \emph{invariant} if it is both positively and
negatively invariant.

\begin{dfn}(\cite[Definition 3.1]{GQT})
Let $F$ be a topological graph. 
\begin{enumerate}
    \item A subgraph $G$ of $F$ is \emph{closed} if $G^1 \subseteq F^1$ and $G^0 \subseteq F^0$ are closed subspaces.
    \item A closed subgraph $G$ of $F$ is \emph{regular} if $G^0$ is negatively invariant in $F^0$ and $r_F^{-1}(G^0) \subseteq G^1$.
\end{enumerate}
\end{dfn}

\begin{remark}
Let \(G\) be a closed subgraph of \(F\). Then it is easy to see that
\[
r_F^{-1}(G^0)\subseteq G^1
\iff
G^0 \text{ is positively invariant in } F^0
\text{ and }
G^1=G^0F^1G^0,
\]
where
\[
G^0F^1G^0
:=
\{e\in F^1:s_F(e),r_F(e)\in G^0\}.
\]
Consequently,
\[
G \text{ is regular in } F
\iff
G^0 \text{ is invariant in } F^0
\text{ and }
G^1=G^0F^1G^0.
\]
\end{remark}

\begin{dfn}(\cite[Definition 2.6]{GQT})
A \emph{factor map} \( m : E \to F \) between topological graphs \(E\) and \(F\) is a pair of proper continuous maps 
\[
m_1 : E^1 \to F^1, \qquad m_0 : E^0 \to F^0
\]
such that:
\begin{enumerate}
    \item[(F1)] For all \( e \in E^1 \),
    \[
    r_F(m_1(e)) = m_0(r_E(e)) \quad \text{and} \quad s_F(m_1(e)) = m_0(s_E(e)).
    \]

    \item[(F2)] If \( x \in F^1 \) and \( v \in E^0 \) satisfies \( r_F(x) = m_0(v) \), then there exists a unique element \( e \in E^1 \) such that
    \[
    m_1(e) = x \quad \text{and} \quad r_E(e) = v.
    \]
\end{enumerate}
\end{dfn}

Let $E$ and $F$ be topological graphs, $G$ a closed subgraph of $F$ such that $r_F^{-1}(G^0) \subseteq G^1$, and $m_0: G^0 \to E^0$, $m_1: G^1 \to E^1$ proper continuous maps satisfying (F1). Consider the adjunction spaces
\[
E^0 \cup_{m_0} F^0, \quad E^1 \cup_{m_1} F^1
\]
with the canonical quotient maps
\[
q_0: E^0 \sqcup F^0 \to E^0 \cup_{m_0} F^0, \quad
q_1: E^1 \sqcup F^1 \to E^1 \cup_{m_1} F^1.
\]
Then, the \emph{adjunction graph} $E \cup_m F$ is defined by
\[
(E \cup_m F)^0 := E^0 \cup_{m_0} F^0, \qquad (E \cup_m F)^1 := E^1 \cup_{m_1} F^1,
\]
with source and range maps
\[
s_\cup(e) := 
\begin{cases}
s_E(e), & e \in E^1, \\
p_0(s_F(e)), & e \in F^1 \setminus G^1,
\end{cases}
\qquad
r_\cup(e) :=
\begin{cases}
r_E(e), & e \in E^1, \\
r_F(e), & e \in F^1 \setminus G^1,
\end{cases}
\]
where $p_0: F^0 \to (E \cup_m F)^0$ is the canonical map given by the composition
\[
p_0: F^0 \hookrightarrow E^0 \sqcup F^0 \xrightarrow{q_0} E^0 \cup_{m_0} F^0.
\]
By \cite[Proposition 3.9]{GQT},  the adjunction graph $E \cup_m F$ is a topological graph. 
In the special case where $m$ is injective, we write 
$$E \cup F := E \cup_m F$$ 
and refer to it as the \emph{union graph}. 
In this case, we write $E \cap F:=G$ and call it 
\emph{intersection graph}.

\begin{dfn}(\cite[Definition 2.6]{GQT})
A  factor map \( m : E \to F \)  is called {\it regular} if
\begin{enumerate}
    \item[(F3)] $m_0(E^0_{sing}) \subseteq F^0_{sing}.$
\end{enumerate}
\end{dfn}

Let $E$ and $F$ be topological graphs, and let $G$ be a regular closed subgraph of $F$. 
Suppose that $m: G \to E$ is a regular factor map.
 If $m$ is injective, then by \cite[Proposition 3.3]{GQT}, $G = E \cap F$ is a regular closed subgraph of $E$.  
Hence, by \cite[Proposition 3.14]{GQT}, both $E$ and $F$ are regular closed subgraphs of $E \cup F$.  
In this case, we call $E \cup F$ the \emph{regular union graph} (\cite[Definition 3.5]{GQT}).

Therefore, we obtain the following pushout diagram in the category of topological graphs and regular factor maps:

\medskip 
\[
\begin{tikzpicture}[scale=0.9]
  % nodes
  \node (E)  at (-6,2) {$E \cup F$};
  \node (F1) at (-8,0) {$E$};
  \node (F2) at (-4,0) {$F$};
  \node (F0) at (-6,-2) {$E \cap F$.};

  % arrows (labels only on lower arrows)
  \draw[->] (F1) -- (E);
  \draw[->] (F2) -- (E);
  \draw[->] (F0) -- node[midway,left]  {} (F1);
  \draw[->] (F0) -- node[midway,right] {} (F2);
\end{tikzpicture}
\]

\begin{remark}\label{regular inclusion and singular sets}
Let \(m:G\to E\) be an injective factor map, and identify \(G\) with
its image \(m(G)\subseteq E\). Then condition {\rm (F2)} is equivalent to
\[
r_E^{-1}(G^0)\subseteq G^1,
\]
which, in turn, is equivalent to
\[
G^0 \text{ is positively invariant in } E^0
\text{ and }
G^1=G^0E^1G^0.
\]
Under these equivalent conditions, \(G^0\) is positively invariant in
\(E^0\), and hence \cite[Proposition~2.2]{Ka3} shows that condition
{\rm (F3):}
\[
G^0_{\mathrm{sing}}\subseteq E^0_{\mathrm{sing}}
\]
is equivalent to the negative invariance of \(G^0\) in \(E^0\).

Consequently, the injective factor map \(m\) is regular if and only if
\[
G^0 \text{ is invariant in } E^0
\text{ and }
G^1=G^0E^1G^0.
\]
\end{remark}

\begin{lem}\label{lem:closed subgraph}
    If $G$ is a closed regular subgraph of a topological graph $F$, then $\partial G$ is a closed subspace $\partial F$.
\end{lem}

\begin{proof}
    We start by proving that $\partial G\subseteq \partial F$. Clearly $G^{\infty}\subseteq F^{\infty}$. Given $x\in G^*$ with $r(x)\in G^0_{\mathrm{sing}}$, by Remark \ref{regular inclusion and singular sets}, $r(x)\in F^0_{\mathrm{sing}}$. Hence $x\in \partial F$.

    Since $G$ is a closed subgraph of $F$, that $\partial G$ is closed in $\partial F$ follows from Lemma \ref{convergence in topological graphs}.
\end{proof}

\subsubsection{Example III: Pullbacks from Regular Union Graphs} We now examine when the above pushout yields a pullback diagram of 
$C^*$-algebras from the groupoid perspective.  
To this end, we establish the following proposition.

\begin{prop}\label{boundary path decomposition:topological graph}
Let  $E\cup F$ be a regular union graph such that 
\[
E^0_{\mathrm{sing}} \cap (E \cap F)^0_{\mathrm{reg}} \subseteq F^0_{\mathrm{reg}}.
\] Then, we have the following:
\begin{enumerate}
    \item[(1)] $\partial(E\cup F) = \partial E \cup \partial F$,
    \item[(2)] $\partial(E\cap F) = \partial E \cap \partial F$.
\end{enumerate}
\end{prop}

\begin{proof}
\noindent\text{(1):} 
We first show that every path in \(E\cup F\) is a path either in \(E\) or in
\(F\).

We begin with the following claim: no path in \(E\cup F\) can switch from an
edge in \(F^1\setminus (E\cap F)^1\) to an edge in \(E^1\). Suppose, toward a
contradiction, that there exists a path \(x=e_1e_2\cdots \in E^\infty \cup E^*\) and \(k\) such
that
\[
e_k\in F^1\setminus (E\cap F)^1,
\qquad
e_{k+1}\in E^1.
\]
Since \(x\) is a path, we have
\[
r_\cup(e_k)=s_\cup(e_{k+1}).
\]
By the definition of the adjunction graph,
\[
r_\cup(e_k)=r_F(e_k),
\qquad
s_\cup(e_{k+1})=s_E(e_{k+1}),
\]
and hence
\[
r_F(e_k)=s_E(e_{k+1}).
\]
Since \((E\cup F)^0=E^0\cup_{m_0}F^0\) and points are identified only along
\((E\cap F)^0\), it follows that
\[
r_F(e_k)\in (E\cap F)^0.
\]
Because \(E\cap F\) is a regular closed subgraph of \(F\), we have
\[
r_F^{-1}((E\cap F)^0)\subseteq (E\cap F)^1.
\]
Thus \(e_k\in (E\cap F)^1\), contradicting
\(e_k\in F^1\setminus (E\cap F)^1\).

Similarly, since \(E\cap F\) is a regular closed subgraph of \(E\), no path in
\(E\cup F\) can switch from an edge in \(E^1\setminus (E\cap F)^1\) to an edge
in \(F^1\).

Therefore a path in \(E\cup F\) cannot contain both an edge in
\(E^1\setminus (E\cap F)^1\) and an edge in
\(F^1\setminus (E\cap F)^1\). Indeed, if such two edges appeared in the same
path, then at some point the path would have to switch from one side to the
other, which is impossible by the preceding paragraph. Since the edges in
\((E\cap F)^1\) belong to both \(E^1\) and \(F^1\), every path in \(E\cup F\) is
a path either in \(E\) or in \(F\). Hence
\[
(E\cup F)^\infty=E^\infty\cup F^\infty,
\qquad
(E\cup F)^*=E^*\cup F^*.
\]

Now we prove the equality of boundary path spaces.
First, by Lemma~\ref{lem:closed subgraph}, $$\partial E\cup\partial F\subseteq \partial(E\cup F).$$
For the reverse inclusion, let \(x\in\partial(E\cup F)\). If \(x\) is infinite, then by the
above equality
\[
(E\cup F)^\infty=E^\infty\cup F^\infty,
\]
we have \(x\in E^\infty\cup F^\infty\), and hence
\(x\in\partial E\cup\partial F\).

Suppose now that \(x\) is finite, and put \(v=r_\cup(x)\). Since
\(x\in\partial(E\cup F)\), we have
\[
v\in (E\cup F)^0_{\mathrm{sing}}
=
E^0_{\mathrm{sing}}\cup F^0_{\mathrm{sing}}.
\]
Also, by the equality \((E\cup F)^*=E^*\cup F^*\), the path \(x\) belongs to
\(E^*\) or to \(F^*\).

If \(x\in E^*\) and \(v\in E^0_{\mathrm{sing}}\), then
\(x\in\partial E\). Similarly, if \(x\in F^*\) and
\(v\in F^0_{\mathrm{sing}}\), then \(x\in\partial F\).

It remains only to consider the possible mixed cases. Suppose, for instance,
that
\[
x\in E^*,
\qquad
v\in F^0_{\mathrm{sing}}.
\]
Since \(x\in E^*\), we have \(v\in E^0\), and since
\(v\in F^0_{\mathrm{sing}}\subseteq F^0\), we get
\[
v\in E^0\cap F^0=(E\cap F)^0.
\]
We claim that \(x\in (E\cap F)^*\). If \(x\) has length zero, this is
immediate. If \(x=(e_1,\ldots,e_n)\) has positive length, then
\(r_E(e_n)=v\in (E\cap F)^0\). Since \(E\cap F\) is a regular closed subgraph
of \(E\), we have
\[
r_E^{-1}((E\cap F)^0)\subseteq (E\cap F)^1,
\]
and therefore \(e_n\in (E\cap F)^1\). Hence
\(s_E(e_n)\in (E\cap F)^0\). Since
\(r_E(e_{n-1})=s_E(e_n)\), the same argument gives
\(e_{n-1}\in (E\cap F)^1\). Repeating this argument, we obtain
\[
e_1,\ldots,e_n\in (E\cap F)^1.
\]
Thus \(x\in (E\cap F)^*\subseteq F^*\). Since \(v\in F^0_{\mathrm{sing}}\), it
follows that
\[
x\in\partial F.
\]
The other mixed case, namely \(x\in F^*\) and \(v\in E^0_{\mathrm{sing}}\), is
handled in the same way, using that \(E\cap F\) is a regular closed subgraph of
\(F\). Hence \(x\in\partial E\).

Therefore every finite boundary path of \(E\cup F\) belongs to
\(\partial E\cup\partial F\). We have shown that
\[
\partial(E\cup F)\subseteq \partial E\cup\partial F.
\]

Combining the two inclusions, we obtain
\[
\partial(E\cup F)=\partial E\cup\partial F.
\]

\noindent\text{(2):}
The inclusion $\partial(E\cap F)\subseteq \partial E\cap \partial F$ follows from Lemma~\ref{lem:closed subgraph}.

For the converse, let $x\in \partial E\cap \partial F$. Since $x$ is a path in both $E$ and $F$,
every edge of $x$ lies in
$E^1\cap F^1=(E\cap F)^1.$
Hence $x$ is a path in $E\cap F$.
If $x$ is infinite, then
$$
x\in (E\cap F)^\infty\subseteq \partial(E\cap F).
$$
Suppose now that $x$ is finite and put $v=r(x)$. Since
$x\in \partial E\cap \partial F$, we have
$
v\in E^0_{\mathrm{sing}}\cap F^0_{\mathrm{sing}}.
$
We claim that
\[
v\in (E\cap F)^0_{\mathrm{sing}}.
\]
Suppose, to the contrary, that
\[
v\in (E\cap F)^0_{\mathrm{reg}}.
\]
Since $v\in E^0_{\mathrm{sing}}$, the hypothesis
\[
E^0_{\mathrm{sing}}\cap (E\cap F)^0_{\mathrm{reg}}
\subseteq F^0_{\mathrm{reg}}
\]
implies that
\[
v\in F^0_{\mathrm{reg}}.
\]
This contradicts $v\in F^0_{\mathrm{sing}}$. Therefore
$
v\notin (E\cap F)^0_{\mathrm{reg}},
$
and hence
$
v\in (E\cap F)^0_{\mathrm{sing}}.
$
Thus $x\in \partial(E\cap F)$.

Consequently,
$
\partial E\cap \partial F\subseteq \partial(E\cap F).
$
\end{proof}

\begin{remark}\label{symmetric:topological graph}
The condition
\[
E^0_{\mathrm{sing}}\cap (E\cap F)^0_{\mathrm{reg}}
\subseteq F^0_{\mathrm{reg}}
\]
is symmetric in \(E\) and \(F\). Indeed, it is equivalent to
\[
E^0_{\mathrm{sing}}\cap F^0_{\mathrm{sing}}
\cap (E\cap F)^0_{\mathrm{reg}}
=
\emptyset,
\]
which is in turn equivalent to
\[
F^0_{\mathrm{sing}}\cap (E\cap F)^0_{\mathrm{reg}}
\subseteq E^0_{\mathrm{reg}}.
\]
\end{remark}

The next result obtains the pullback theorem for regular union graphs from
\cite[Theorem~4.2]{GQT} by applying the groupoid pullback
Theorem~\ref{thm:disjoint open invariant sets}.

\begin{thm}[Pullback for regular union graphs]\label{thm:regular-union-pullback}
Let \(E\cup F\) be a regular union graph such that
\[
E^0_{\mathrm{sing}}\cap (E\cap F)^0_{\mathrm{reg}}
\subseteq F^0_{\mathrm{reg}}.
\]
Then the corresponding topological graph \(C^*\)-algebras form a pullback
diagram in the category of \(\mathbb T\)-\(C^*\)-algebras and
\(\mathbb T\)-equivariant \( * \)-homomorphisms:
\begin{equation}\label{diagram:topological graph algebras}
\begin{aligned}
\begin{tikzpicture}
  \node (P) at (0,2) {$C^*(E\cup F)$};
  \node (A) at (-2,0) {$C^*(E)$};
  \node (B) at (2,0) {$C^*(F)$};
  \node (C) at (0,-2) {$C^*(E\cap F)$};

  \draw[->] (P) -- (A);
  \draw[->] (P) -- (B);
  \draw[->] (A) -- (C);
  \draw[->] (B) -- (C);
\end{tikzpicture}
\end{aligned}
\end{equation}
where the \(\mathbb T\)-actions are the canonical gauge actions.
\end{thm}

\begin{proof}
Put $H:=E\cup F$.  Let $\Gamma(\partial H,\sigma_H)$
be the Deaconu--Renault groupoid associated with the one-sided shift map
$\sigma_H:\partial H\setminus H^0_{\mathrm{sing}}\to\partial H.$
We equip \(\Gamma(\partial H,\sigma_H)\) with the canonical cocycle
\[
c_H:\Gamma(\partial H,\sigma_H)\to\mathbb Z,
\qquad
c_H(x,k,y)=k.
\]
The restriction of \(c_H\) to each reduction considered below induces the
corresponding gauge action on its groupoid \(C^*\)-algebra.

By Proposition~\ref{boundary path decomposition:topological graph}, we have
$
\partial H=\partial E\cup\partial F
$
and
$
\partial E\cap\partial F=\partial(E\cap F).
$
Define
\[
\Omega_E:=\partial H\setminus\partial E,
\qquad
\Omega_F:=\partial H\setminus\partial F.
\]
Then, by Lemma~\ref{lem:closed subgraph}, the subsets \(\partial E\) and
\(\partial F\) are closed in  \(\partial H\), and it is easy to check that they are
\(\Gamma(\partial H,\sigma_H)\)-invariant.
 Hence \(\Omega_E\) and
\(\Omega_F\) are open and
\(\Gamma(\partial H,\sigma_H)\)-invariant. Moreover,
\[
\Omega_E\cap\Omega_F
=
\partial H\setminus(\partial E\cup\partial F)
=
\emptyset.
\]
We also have
\[
\begin{aligned}
\Gamma(\partial H,\sigma_H)_{\partial E}
\cap
\Gamma(\partial H,\sigma_H)_{\partial F}
&=
\Gamma(\partial H,\sigma_H)_{\partial E\cap\partial F}\\
&=
\Gamma(\partial H,\sigma_H)_{\partial(E\cap F)}.
\end{aligned}
\]

Applying Theorem~\ref{thm:disjoint open invariant sets} to the disjoint open
invariant subsets \(\Omega_E\) and \(\Omega_F\), together with the cocycle
\(c_H\), we obtain a pullback diagram in the category of
\(\mathbb T\)-\(C^*\)-algebras and \(\mathbb T\)-equivariant
\( * \)-homomorphisms:
\[
\begin{tikzpicture}
  \node (P) at (0,2) {$C^*(\Gamma(\partial H,\sigma_H))$};
  \node (A) at (-2.8,0)
  {$C^*(\Gamma(\partial H,\sigma_H)_{\partial E})$};
  \node (B) at (2.8,0)
  {$C^*(\Gamma(\partial H,\sigma_H)_{\partial F})$};
  \node (C) at (0,-2)
  {$C^*(\Gamma(\partial H,\sigma_H)_{\partial(E\cap F)}).$};

  \draw[->] (P) -- (A);
  \draw[->] (P) -- (B);
  \draw[->] (A) -- (C);
  \draw[->] (B) -- (C);
\end{tikzpicture}
\]

We now identify the groupoids occurring in this diagram. Since the shift map
\(\sigma_H\) restricts to the corresponding shift maps on
\(\partial E\), \(\partial F\), and \(\partial(E\cap F)\), respectively,
there are canonical groupoid isomorphisms
\[
\Gamma(\partial H,\sigma_H)_{\partial E}
\cong
\Gamma(\partial E,\sigma_E),
\]
\[
\Gamma(\partial H,\sigma_H)_{\partial F}
\cong
\Gamma(\partial F,\sigma_F),
\]
and
\[
\Gamma(\partial H,\sigma_H)_{\partial(E\cap F)}
\cong
\Gamma(\partial(E\cap F),\sigma_{E\cap F}).
\]
These isomorphisms preserve the canonical cocycles: the restrictions of
\(c_H\) correspond to
\[
c_E(x,k,y)=k,
\qquad
c_F(x,k,y)=k,
\qquad
c_{E\cap F}(x,k,y)=k.
\]
Consequently, the induced \(C^*\)-algebra isomorphisms are
\(\mathbb T\)-equivariant.

Using the canonical gauge-equivariant isomorphisms between topological graph
\(C^*\)-algebras and the \(C^*\)-algebras of their boundary-path groupoids,
we obtain
\[
C^*(\Gamma(\partial H,\sigma_H))
\cong
C^*(H)
=
C^*(E\cup F),
\]
\[
C^*(\Gamma(\partial H,\sigma_H)_{\partial E})
\cong
C^*(E),
\qquad
C^*(\Gamma(\partial H,\sigma_H)_{\partial F})
\cong
C^*(F),
\]
and
\[
C^*(\Gamma(\partial H,\sigma_H)_{\partial(E\cap F)})
\cong
C^*(E\cap F).
\]

Therefore, after applying these \(\mathbb T\)-equivariant isomorphisms to the
above diagram, we obtain the claimed diagram (\ref{diagram:topological graph algebras})
as a pullback in the category of \(\mathbb T\)-\(C^*\)-algebras and
\(\mathbb T\)-equivariant \( * \)-homomorphisms.
\end{proof}

\end{document}